\documentclass[11pt, reqno]{amsart}
\usepackage{amsmath, amscd,amsxtra,amssymb,mathrsfs, bbm}
\usepackage{stmaryrd}
\usepackage[margin=2cm]{geometry}
\usepackage{enumerate}
\usepackage{tikz-cd}

\newtheorem{theorem}{Theorem}[section]

\newtheorem{lemma}[theorem]{Lemma}

\newtheorem*{theorem*}{Theorem}

\theoremstyle{definition}
\newtheorem{definition}[theorem]{Definition}

\DeclareMathAlphabet{\mathpzc}{OT1}{pzc}{m}{it}

\DeclareMathOperator{\Hom}{\mathsf{Hom}}

\DeclareMathOperator{\Ext}{\mathsf{Ext}}

\DeclareMathOperator{\Sym}{\mathsf{Sym}}

\input xy
\xyoption{all}

\newcommand{\kO}{\mathcal{O}}

\newcommand{\kN}{\mathcal{N}}

\DeclareMathOperator{\cL}{\mathsf{L}}
\DeclareMathOperator{\cR}{\mathsf{R}}

\usepackage{hyperref}

\begin{document}
\title[Tangent bundle of a manifold of K$3^{[2]}$-type is rigid]{Tangent bundle of a manifold of K$3^{[2]}$-type is rigid}
\author{Volodymyr Gavran}
\email{vlgvrn@gmail.com}
\maketitle

\begin{abstract}
We prove that the tangent bundle of a manifold of K$3^{[2]}$-type is rigid.
\end{abstract}
\section{Introduction}
M.Verbitsky in the work \cite{verb} showed that for a hyperholomorphic vector bundle $F$ on a hyperk\"ahler manifold $X$ there are no obstructions for stable deformations of $F$ besides the Yoneda pairing on $\Ext^1(F, F)$. Moreover, he proved the existence of a canonical hyperk\"ahler structure on the reduction of the coarse moduli space of stable deformations of $F$. If $S$ is a K3 surface then it is known that the Hilbert scheme $S^{[n]}$ is a hyperk\"ahler manifold and the tangent bundle $T_{S^{[n]}}$ is a hyperholomorphic bundle on $S^{[n]}$. Thus, the investigation of the deformation space of the bundle $T_{S^{[n]}}$ is a very natural and interesting question from the point of view of hyperkahler geometry. This question also appeared in \cite{charles} in the context of the Lefschetz standard conjecture for hyperk\"ahler manifolds. It was mentioned there without a proof that for $n = 2$ the tangent bundle might actually be rigid. In the present note we confirm this statement by proving the following theorem.
\begin{theorem}
\label{main} 
Let $X$ be a manifold of K$3^{[2]}$-type. Then the tangent bundle $T_X$ is infinitesimally rigid, i.e. $H^1(X, {\mathcal End}(T_X))=0$. \end{theorem}
The proof of this statement is given in sections 3 and 4. It follows from explicit computations in the case when $X$ is the second Hilbert scheme of a K3 surface and then from application of general results from the theory of hyperholomorphic bundles \cite{verb}. For $n>2$ the question about deformations of $T_{S^{[n]}}$ seems to be much more difficult due to more complicated geometry of the corresponding Hilbert scheme and thus should be considered separately.

\medskip
\noindent\textbf{Acknowledgements.} The author is grateful to Christopher Brav and Misha Verbitsky for helpful discussions and to Fran\c{c}ois Charles for suggestions.
\section{Hilbert square}
For a smooth projective surface $S$ the Hilbert scheme of length-2 subschemes of $S$ is denoted by $S^{[2]}$. Let $\Delta:S\hookrightarrow S\times S$ be the diagonal embedding, $p_1,p_1':S\times S\to S$ be the projections onto the first and the second component and $\sigma:Z \stackrel{\mathsf{def}}= \mathsf{Bl}_\Delta(S\times S)\to S\times S$ be the blowup of $S\times S$ in $\Delta$. The natural action of the symmetric group $\mathfrak{S}_2$ on $S\times S$ extends to an action on $Z$ and the Hilbert square $S^{[2]}$ is the quotient of $Z$ by this action. By $q_2$ we denote the corresponding quotient map $Z\to S^{[2]}$. Let $j:E\hookrightarrow Z$ be the exceptional divisor of $\sigma$. Recall that $E\cong\mathbb{P}(T_S)$ is a projective bundle over $S$ and we have the relative Euler exact sequence
\begin{equation}\label{eul}
0\to\Omega_{E/S}\to\pi^*\Omega_S(-1)\to\kO_E\to0.
\end{equation}
Also, by $\iota:D\hookrightarrow S^{[2]}$ we denote the isomorphic image of $E$ by $q_2$. The divisor $D$ is precisely the locus parametrizing non-reduced subschemes of $S$ of length two.

Put $q_1 := p_1\circ\sigma$ and $q_1' := p_1'\circ\sigma$. The following diagram depicts the relationship among all the natural maps between the varieties that we mentioned:
\begin{equation}
\label{equation:notationHilb2}
\begin{tikzcd}
& E\cong\mathbb{P}(T_S) \arrow[r, hook, "j"] \arrow[dl, swap, "\pi"] & Z \arrow[dl, swap, "\sigma"] \arrow[dr, "q_2"] \arrow[d, "q_1"]\\ 
S \arrow[r, hook, "\Delta"] & S\times S \arrow[d, "p_1'"] \arrow[r, "p_1"] & S & S^{[2]} & D. \arrow[l, hook', swap, "\iota"]\\
& S &
\end{tikzcd}
\end{equation}
Here $\pi:E\to S$ is the projective bunde map and the equality $\sigma\circ j = \Delta\circ\pi$ yields $q_1\circ j = \pi$.

Note that $Z$ is isomorphic to the universal closed subscheme $\mathcal{Z}: = \{(x, \xi)\, |\, x\in\mathsf{Supp}(\xi)\}$ in $S\times S^{[2]}$ and for any coherent sheaf $F$ over $S$ there is the incidence exact sequence
\begin{equation}\label{incidence}
0\to q_1^*F(-E)\to q_2^*F^{[2]}\to q_1'^{*}F\to0,
\end{equation}
where $F^{[2]}$ is the image of $F$ under the tautological functor $q_{2*}q_1^*:\mathsf{Coh}(S)\to\mathsf{Coh}(S^{[2]})$ (see \cite[p. 193]{lehn}).

Recall that $q_{2*}\kO_{Z}\cong\kO_{S^{[2]}}\oplus L^{-1}$, where the line bundle $L^{-1}$ is the eigenspace to the eigenvalue $-1$ of
the cover involution. Moreover, $L^{\otimes2}\cong\kO_{S^{[2]}}(D)$ and $q_2^* L\cong\kO_Z(E)$. Note that $q_{2*}j_*\kO_E\cong\iota_*\kO_D$ and $q_2^*\iota_*\kO_D\cong\kO_{2E}$.

There is an exact sequence
\begin{equation}
\label{pullbackCotangent}
0\to q_2^*\Omega_{S^{[2]}}(E)\to\Omega_Z(E)\to j_*\kO_E\to0
\end{equation}
and an isomorphism
\[\Omega_{S^{[2]}}\cong q_{2*}(\kN^\vee_{Z/S\times S^{[2]}}(E)).\]
The sequence \eqref{pullbackCotangent} implies that $\omega_{q_2}\cong\kO_Z(E)$, hence the right adjoint functor to $q_{2*}$ is $q_2^!(-)=q_2^*(-)\otimes\kO_Z(E)$. 

Now we write down the exact sequence defining the cotangent bundle on $S^{[2]}$. Putting left non-zero arrow of the sequence \eqref{pullbackCotangent} together with the conormal exact sequence of the embedding $Z\hookrightarrow S\times S^{[2]}$ twisted by $E$ into a commutative diagram
\begin{equation}
\label{cotangentHilb2}
\begin{tikzcd}
& 0 \arrow[d] \arrow[r] & q_2^*\Omega_{S^{[2]}}(E) \arrow[r, "\sim"] \arrow[d] & q_2^*\Omega_{S^{[2]}}(E) \arrow[d]\\
0 \arrow[r] & \kN^\vee_{Z/S\times S^{[2]}}(E) \arrow[r]    & q_1^*\Omega_S(E)\oplus q_2^*\Omega_{S^{[2]}}(E)   \arrow[r]      & \Omega_Z(E)
\end{tikzcd}
\end{equation}
and applying the snake lemma we obtain the exact sequence
\begin{equation}
\label{snake}
0 \longrightarrow \kN^\vee_{Z/S\times S^{[2]}}(E) \longrightarrow q_1^*\Omega_S(E)\longrightarrow j_*\kO_E\longrightarrow 0.
\end{equation}
After pushing forward \eqref{snake} along $q_2$ we obtain the exact sequence
\begin{equation}\label{cotangent}
0\to\Omega_{S^{[2]}}\to\Omega_S^{[2]}\otimes L\to\iota_*\kO_D\to0.
\end{equation}
\section{Computation for the Hilberst square of K3 surface}
From now on we assume that $S$ is a K3 surface. We fix some isomorphism $T_S\stackrel{\simeq}\longrightarrow\Omega_S$. The isomorphism $\omega_S\cong\kO_S$ yields $\omega_Z\cong\kO_Z(E)$. From the Euler sequence \eqref{eul} it follows that $\Omega_{E/S}\cong\kO_E(-2)$. From stability of $\Omega_S$ we have that $\Hom(\Omega_S, \Omega_S)\cong\mathbb{C}$. The latter implies that $H^0(S, \Sym^2(\Omega_S)) = 0$. Also, we will use the equality $H^0(S, \Omega_S) = 0$ which by \cite[Remark 3.19]{krug} and by stability of $\Omega_S$ implies that $\Hom(\Omega_S^{[2]}, \Omega_S^{[2]})\cong\mathbb{C}$.

To prove the Theorem \ref{main} in this case it is enough to show the following two equalities
\begin{equation}
\label{firstVanishing}
\Ext^1(\Omega_S^{[2]}\otimes L, \Omega_{S^{[2]}})=0,
\end{equation}
\begin{equation}
\label{secondVanishing}
\Ext^2(\iota_*\kO_D, \Omega_{S^{[2]}})=0.
\end{equation}

\begin{lemma}\label{usefulEqualities} The following equalities hold
\begin{enumerate}[(i)]
  \item $\Hom(q_1^*\Omega_S(E), j_*\kO_E) \cong \Hom(q_1^*\Omega_S(E), q_1^*\Omega_S(E)|_E) \cong \Hom(q_1^*\Omega_S(E)|_E, j_*\kO_E) \cong\mathbb{C},$
  \item $\Ext^1(\iota_*\kO_D, \iota_*\kO_D) = 0,$
  \item $\Hom(\Omega_S^{[2]}\otimes L, \iota_*\kO_D)\cong \mathbb{C},$
  \item $\Ext^2(\kO_{2E}, q_1^*\Omega_S) = 0,$
  \item $\Ext^k(q_2^*\Omega_S^{[2]}, q_1^*\Omega_S(-E))=0$ for $k=0,1$,
  \item $\Ext^k(q_2^*\Omega_S^{[2]}, j_*\Omega_{E/S}(-E))=0$ for $k=0,1$,
  \item $\Ext^2(\kO_{2E}, j_*\Omega_{E/S}) = 0$.
\end{enumerate}
\end{lemma}
\begin{proof}
All listed equalities are straightforward consequences of standard adjunctions, properties of the blow-up and projective bundle map $E\to S$, so we only give sketched proofs.

\noindent (i) We have
\begin{align*}\Hom(q_1^*\Omega_S(E), j_*\kO_E)&\cong\Hom(p_1^*\Omega_S, \sigma_*j_*\kO_E(-E))\\
&\cong\Hom(p_1^*\Omega_S, \Delta_*\pi_*\kO_E(1))\\
&\cong\Hom(p_1^*\Omega_S, \Delta_*T_S)\\
&\cong\Hom(\Omega_S, T_S)\\
&\cong\mathbb{C}.
\end{align*}

By the projection formula we have that $\Hom(q_1^*\Omega_S(E), q_1^*\Omega_S(E)|_E)\cong\Hom(\Omega_S, \Omega_S)\cong\mathbb{C}$. Finally, $\Hom(\pi^*\Omega_S(E), \kO_E)\cong\Hom(\Omega_S, T_S)\cong\mathbb{C}$. From the fact that $q_1\circ j = \pi$ and since the functor $j_*$ is fully faithful on the level of abelian categories, we get $\Hom(q_1^*\Omega_S(E)|_E, j_*\kO_E)\cong\mathbb{C}$.

\noindent (ii) Using that $\cL\!j^*\kO_{2E} \cong \kO_E\oplus\kO_E(-2E)[1]$ and $\cR\!\pi_*\kO_E(-2)\cong\omega^\vee_S[-1]$, by the adjunction and the projection formula we have
\begin{align*}\Ext^k(\iota_*\kO_D,\iota_*\kO_D)&\cong\Ext^k(\iota_*\kO_D, q_{2*}j_*\kO_E)\\
&\cong\Ext^k(q_2^*\iota_*\kO_D, j_*\kO_E)\\
&\cong\Ext^k(\kO_{2E}, j_*\kO_E)\\
&\cong\Ext^k(\kO_E, \kO_E)\oplus\Ext^{k-1}(\kO_E, \kO_E(2E))\\
&\cong H^k(S,\kO_S)\oplus H^{k-2}(S, \omega^\vee_S).
\end{align*}
Hence $\Ext^1(\iota_*\kO_D,\iota_*\kO_D) = 0$.

\noindent (iii) Since $\Hom(q_1^*\Omega_S, j_*\kO_E)= 0$ and $\Hom(q_1'^*\Omega_S, j_*\kO_E(-E))\cong\Hom(\Omega_S, T_S)\cong\mathbb{C}$, from the exact sequence \eqref{incidence} with $F=\Omega_S$ we have that $\Hom(q_2^*\Omega_S^{[2]}, j_*\kO_E(-E))\cong\mathbb{C}$. Then $\Hom(\Omega_S^{[2]}\otimes L, \iota_*\kO_D)\cong\Hom(q_2^*\Omega_S^{[2]}, j_*\kO_E(-E))\cong\mathbb{C}$. 

\noindent (iv) Applying $\sigma_*$ to the exact sequence
\begin{equation}
\label{exactSequenceO2E}
0\to j_*\kO_E\to\kO_{2E}(E)\to j_*\kO_{E}(E)\to0
\end{equation}
and using the equality $\cR\!\pi_*\kO_E(-1)=0$ we obtain that $\cR\!\sigma_*\kO_{2E}(E)\cong\kO_\Delta$. Thus $\Ext^2(q_1^*\Omega_S, \kO_{2E}(E))\cong H^2(S, T_S) = 0$. The assertion then follows from the Serre duality.

\noindent (v) Applying adjunctions $q_1^*\dashv q_{1*}$ and $q_{2*}\dashv q_2^!$ we obtain 
$\Ext^k(q_2^*\Omega_S^{[2]}, q_1^*\Omega_S(-E))\cong\Ext^k(q_1^*\Omega_S, q_2^*\Omega_S^{[2]})
$.
Since $\cR\!q_{1*}q_1'^{*}\Omega_S\cong H^1(S, \Omega_S)\otimes\kO_S[-1]$ we have that $\Ext^k(q_1^*\Omega_S,q_1'^*\Omega_S)=0$ for $k=0,1$. From the exact sequence
\[0\to I_\Delta\to\kO_{S\times S}\to\kO_\Delta\to0\]
and the condition $H^1(S, \kO_S) = 0$ we obtain $\cR\!p_{1*}I_\Delta=\kO_S[-2]$. This implies that $\Ext^k(q_1^*\Omega_S,q_1^*\Omega_S(-E))\cong\Ext^k(\Omega_S,\Omega_S\otimes\cR\!p_{1*}I_\Delta)=0$ for $k = 0,1$. Now, applying $\Hom(q_1^*\Omega_S, -)$ to the incidence exact sequence \eqref{incidence} with $F = \Omega_S$, we obtain the desired statement.

\noindent (vi) Applying $\sigma_*$ to the sequence \eqref{exactSequenceO2E} twisted by $E$, we obtain the isomorphism $\cR\!\sigma_*\kO_{2E}(2E)\cong\kO_\Delta[-1]$. Together with the isomorphism $\Omega_{E/S}\cong\kO_E(-2)$ it gives
\[\Ext^k(q_2^*\Omega_S^{[2]}, j_*\Omega_{E/S}(-E))\cong\Ext^k(\Omega_S^{[2]}, \iota_*\kO_D\otimes L)\]
\[\cong\Ext^k(q_1^*\Omega_S, \kO_{2E}(2E))\cong\Ext^{k - 1}(\Omega_S, \kO_S) = 0\]
for $k = 0, 1$.

\noindent (vii) We have that $\Ext^2(\kO_{2E}, j_*\Omega_{E/S})\cong H^1(S, \omega_S^\vee)\oplus H^0(S, \Sym^2(T_S)) = 0.$
\end{proof}
From Lemma \ref{usefulEqualities}(1) we have that the map $q_1^*\Omega_S(E)\to j_*\kO_E$ in the exact sequence \eqref{snake} factors as the composition of natural maps
\begin{equation}\label{composition}
q_1^*\Omega_S(E)\longrightarrow q_1^*\Omega_S(E)|_E\longrightarrow j_*\kO_E,
\end{equation}
where the second map is the pushforward along $j$ of the quotient map $\pi^*\Omega_S(-1)\to\kO_E$ in the Euler exact sequence. 

Consider the maps
\begin{equation}\label{firstMap}
\alpha_k:\Ext^k(\Omega_S^{[2]}, \Omega_S^{[2]})\longrightarrow\Ext^k(\Omega_S^{[2]}\otimes L, \iota_*\kO_D), \,\,\, k = 0, 1,
\end{equation}
\begin{equation}\label{secondMap}
\beta_2:\Ext^2(\iota_*\kO_D, \Omega_S^{[2]}\otimes L)\longrightarrow\Ext^2(\iota_*\kO_D,\iota_*\kO_D),
\end{equation}
coming from the exact sequence \eqref{cotangent}. By the adjunction $q_2^*\dashv q_{2*}$ and factorization \eqref{composition} the map $\alpha_k$ can be written as the composition
\begin{equation}\label{extifact}
\Ext^k(q_2^*\Omega_S^{[2]}, q_1^*\Omega_S)\to\Ext^k(q_2^*\Omega_S^{[2]}, q_1^*\Omega_S|_E)\to\Ext^k(q_2^*\Omega_S^{[2]}, j_*\kO_E(-E)).
\end{equation}
From assertions (v) and (vi) of Lemma \ref{usefulEqualities} it follows that both maps in \eqref{extifact} are injective, thus $\alpha_0$ and $\alpha_1$ are injective as well. Moreover, by Lemma \ref{usefulEqualities}(iii) we get that $\alpha_0$ is an isomorphism since it is a map between one-dimensional vector spaces. This implies the equality $\eqref{firstVanishing}$. 

Similarly, we now decompose $\beta_2$ as
\begin{equation}\label{ext2factor}
\Ext^2(\kO_{2E}, q_1^*\Omega_S(E))\to\Ext^2(\kO_{2E}, q_1^*\Omega_S(E)|_E)\to\Ext^2(\kO_{2E}, j_*\kO_E).
\end{equation}
Lemma \ref{usefulEqualities}(iv) implies the injectivity of the first map in \eqref{ext2factor}. The injectivity of the second map follows from Lemma \ref{usefulEqualities}(vii). This shows that $\beta_2$ is injective, which together with Lemma \ref{usefulEqualities}(ii) gives the vanishing \eqref{secondVanishing}. 

\section{General case}
Let $X$ be an irreducible holomorphic symplectic manifold and $\mathcal{H} = (I, J, K)$ be the corresponding hyperk\"ahler structure. For any triple $a,b,c\in\mathbb{R}$ such that $a^2 + b^2 + c^2 = 1$ the operator $L := aI + bJ + cK$ defines a complex structure on $X$. Such a complex structure $L$ is called \emph{induced by the hyperk\"ahler structure}. The space $Q_{\mathcal{H}}$ of all induced complex structures of $\mathcal{H}$ is isomorphic to $\mathbb{C}P^1$ and is called \emph{the twistor line} of $\mathcal{H}$. Denote by $\mathsf{Comp}_X$ the coarse moduli space of complex structures on $X$. Then for each hyperk\"ahler structure we have an embedding $Q_{\mathcal{H}}\subset\mathsf{Comp}_X$.
\begin{definition}
\emph{A twistor path} in $\mathsf{Comp}_X$ is a collection of consecutively intersecting twistor lines $Q_0,...,Q_n\subset\mathsf{Comp}_X$. Two points $I, I'\in\mathsf{Comp}_X$ are called \emph{equivalent} if there exists a twistor path $\gamma = Q_0,...,Q_n$ such that $I\in Q_0$ and $I'\in Q_n$. The path $\gamma$ is then called \emph{a connecting path} of $I$ and $I'$.
\end{definition}

\begin{theorem}\cite[Theorem 3.2]{verb2}\label{connectinPath}
Any two points $I, I'\in\mathsf{Comp}_X$ are equivalent. 
\end{theorem}
Now we recall the definition of a hyperholomorphic bundle over $X$. 
\begin{definition}
Let $F$ be a holomorphic vector bundle over $(X, L)$ with a Hermitian connection $\nabla$ on $F$. The connection $\nabla$ is called \emph{compatible with a holomorphic structure} if $\nabla_v(\xi) = 0$ for any holomorphic section $\xi\in F$ and any antiholomorphic tangent vector $v$. If there exists a holomorphic structure compatible with the given Hermitian connection $\nabla$, then this connection is called
\emph{integrable}. The connection $\nabla$ is called \emph{hyperholomorphic} if it is integrable for any complex structure induced by the hyperk\"ahler structure. Then $F$ is called a \emph{hyperholomorphic bundle}.
\end{definition}
For an induced complex structure $L$ denote by $H^*_L(X, F)$ the holomorphic cohomologies of $F$ with respect to $L$. We mention the following important property of hyperholomorphic bundles.
\begin{theorem}\cite[Corollary 8.1]{verb}
\label{dimensionCohomology}
Let $F$ be a hyperholomorphic vector bundle. Then for any $i\geqslant0$ the dimension of the space $H^i_L(X, \mathcal{E}nd(F))$ is independent of an induced complex structure $L$.
\end{theorem}

Note that the tangent bundle $T_X$ equipped with the Levi-Civita connection is always hyperholomorphic (see \cite[Example 2.9(i)]{verb3}). By Theorem \ref{connectinPath}, for any deformation $X' = (X, I')$, $I'\in\mathsf{Comp}_X$ of $(X, I)$ there exists a twistor path $\gamma$ connecting $I'$ and $I$. Since $T_X$ is hyperholomorphic, the dimension of the cohomology space $H^1(X, \mathcal{E}nd(T_X))$ is constant along $\gamma$ by Theorem \ref{dimensionCohomology}. In the case when $X$ is a manifold of K3$^{[2]}$-type this dimension is equal to zero by the result of Section 3. This proves Theorem \ref{main}.
\section{References}
\renewcommand\refname{}

\end{document}